\documentclass[12pt,letterpaper]{article}

\usepackage{amssymb,amsfonts,amsmath,latexsym,amsthm}
\usepackage[usenames]{color}
\usepackage{bm}
\usepackage{url}
\usepackage{hyperref}
\usepackage[authoryear,round]{natbib}

\theoremstyle{plain}
\newtheorem{theorem}{Theorem}[section]
\newtheorem{corollary}[theorem]{Corollary}
\newtheorem{lemma}[theorem]{Lemma}
\newtheorem{proposition}[theorem]{Proposition}

\newtheorem{assumption}[theorem]{Assumption}
\theoremstyle{definition}

\theoremstyle{remark}

\newtheorem*{remark}{Remark}
\numberwithin{equation}{section}

\newcommand{\boldface}{\bf}

\newcommand{\union}{\cup}
\newcommand{\n}{\cap}
\newcommand{\I}{\bigcap}
\newcommand{\U}{\bigcup}
\renewcommand{\implies}{\Rightarrow}

\renewcommand{\emptyset}{\varnothing}

\renewcommand{\a}{\alpha}
\newcommand{\T}{{\bm\theta}}
\renewcommand{\phi}{\varphi}
\newcommand{\branch}[4]{
\left\{
	\begin{array}{ll}
		#1  & \mbox{if } #2 \\
		#3 & \mbox{if } #4
	\end{array}
\right.
}
\newcommand{\C}{\mathcal{C}}
\newcommand{\A}{\mathcal{A}}
\newcommand{\B}{\mathcal{B}}
\newcommand{\X}{\mathcal{X}}
\newcommand{\Y}{\mathcal{Y}}
\renewcommand{\L}{\mathcal{L}}
\newcommand{\RC}{\mathcal{R}}
\newcommand{\UU}{\mathcal{U}}
\newcommand{\V}{\mathcal{V}}

\newcommand{\AI}{\mathcal{A}^\infty}
\newcommand{\XI}{\mathcal{X}^\infty}
\newcommand{\PI}{P^\infty_\theta}
\newcommand{\Xhat}{\widehat\X}

\DeclareSymbolFont{AMSb}{U}{msb}{m}{n}
\DeclareMathSymbol{\N}{\mathbin}{AMSb}{"4E}
\DeclareMathSymbol{\Z}{\mathbin}{AMSb}{"5A}
\DeclareMathSymbol{\R}{\mathbin}{AMSb}{"52}
\DeclareMathSymbol{\Q}{\mathbin}{AMSb}{"51}
\DeclareMathSymbol{\PP}{\mathbin}{AMSb}{"50}
\DeclareMathSymbol{\E}{\mathbin}{AMSb}{"45}
\begin{document}

\title{A detailed treatment of Doob's theorem}
\author{Jeffrey W. Miller\\
Harvard University, Department of Biostatistics}
% \author{Jeffrey W. Miller\thanks{The author gratefully acknowledges support from .}\hspace{.2cm}\\

% \keywords{}
% \subjclass[2000]{Primary: 06B10; Secondary: 06D05}
% \date{}

\maketitle

\begin{abstract}
Doob's theorem provides guarantees of consistent estimation and posterior consistency under very general conditions.  Despite the limitation that it only guarantees consistency on a set with prior probability 1, for many models arising in practice, Doob's theorem is an easy way of showing that consistency will hold almost everywhere.  In this article, we give a detailed proof of Doob's theorem.
\end{abstract}

\section{Introduction}

\citet{Doob_1949} established a remarkable theorem on consistency in the Bayesian setting.
Roughly, Doob's theorem says that with probability $1$, if the true parameter $\theta_0\in\Theta$ is drawn from the prior $\Pi$, and $X_1,\dotsc,X_n$ i.i.d.\ $\sim P_{\theta_0}$ for an identifiable model $\{P_\theta:\theta\in\Theta\}$, then
\begin{enumerate}
\item the posterior mean of $g(\theta)$ converges to $g(\theta_0)$ for any $g\in L^1(\Pi)$, and
\item the posterior distribution of $\theta$ is consistent, i.e., the posterior concentrates in neighborhoods of the true parameter $\theta_0$,
\end{enumerate}
as $n\to\infty$, under very general conditions. The theorem is stunning in its generality and elegance. The heart of the proof is a direct application of martingale convergence, however, the bulk of the proof involves showing an essential and nontrivial measurability result: that $\theta_0$ is a \emph{measurable} function of $(X_1,X_2,\dotsc)$, i.e., that there is a measurable function $f$ such that with probability $1$, $\,\theta_0 = f(X_1,X_2,\ldots)$. In other words, given infinite data, the true parameter can be recovered in a measurable way.

Doob's theorem has been criticized on the grounds that it only guarantees consistency on a set with prior probability $1$, and thus, if for example one chose a prior with all of its mass on a single point $\theta^*$, then the theorem would apply but consistency would fail everywhere except at $\theta^*$ \citep{barron1999consistency}. While it is true that the theorem is not very useful for a poorly chosen prior, this misses the point, which is that for a well-chosen prior, consistency can be guaranteed on a very large set. For example, if $\Theta$ is a countable union of finite-dimensional spaces, then one can choose a prior with a density with respect to Lebesgue measure that is positive on all of $\Theta$, and Doob's theorem can be used to guarantee consistency everywhere except on a set of Lebesgue measure zero, assuming identifiability. This approach was used by \citet{Nobile_1994} to prove consistency for finite mixture models with a prior on the number of components, under very general conditions. 

However, this approach has the weakness that one cannot say for any given point whether consistency will hold, and in this respect the criticism of Doob's theorem is valid.  Further, Doob's theorem is less useful when $\theta$ is infinite-dimensional, since the set on which consistency may fail can be very large, even for a seemingly well-chosen prior \citep{freedman1963asymptotic}.  These limitations were the impetus for the beginnings of the modern approach to studying the consistency of nonparametric Bayesian procedures \citep{schwartz1965bayes}.  Nonetheless, for many models arising in practice, Doob's theorem is an extremely easy method of guaranteeing that consistency will hold almost everywhere.

The original justification given by \citet{Doob_1949} was hardly even a sketch, providing only the outline of a proof.  More complete proofs of Doob's theorem have been provided for special cases: \citet{van1998asymptotic} (Theorem 10.10) assumes $\Theta$ is a Euclidean space and proves the part of the theorem regarding posterior concentration. \citet{Ghosh} (Theorem 1.3.2) prove the posterior concentration part of the theorem, but they leave many details to the reader. \citet{ghosal2017fundamentals} (Theorem 6.9, 6.10) provide a proof for the general case, but their treatment is brief and many details are left to the reader.

\citet{lijoi2004extending} prove an interesting result related to Doob's theorem in the context of posterior inference for a density function $f_0$ given i.i.d.\ samples $X_1,X_2,\ldots\sim f_0$. They show that if the prior has full Hellinger support then the posterior concentrates in Hellinger neighborhoods of an ``essentially unique'' random density $\tilde g$ with the property that the data are conditionally i.i.d.\ given $\tilde g$. 
% However, the relationship between $\tilde g$ and the true density $f_0$ is not clear.

The purpose of this note is to provide a complete and detailed proof of Doob's theorem in full generality.
The rest of the paper is organized as follows.
In Section \ref{section:theorem}, we state the main results, and Section \ref{section:main-proofs} contains the proofs of the main results. 
In Section \ref{section:recovery}, we prove the key measurability result that is used in the proofs.
In Section \ref{section:details}, for completeness, we prove some supporting results.

\section{Main results}
\label{section:theorem}

Let $\Xhat$ and $\widehat\Theta$ be complete separable metric spaces. Let $\X\subseteq\Xhat$ and $\Theta\subseteq\widehat\Theta$ be Borel measurable subsets, and let $\A$ and $\B$ be the induced Borel $\sigma$-algebras of $\X$ and $\Theta$, respectively. For each $\theta\in\Theta$, let $P_\theta$ be a probability measure on $(\X,\A)$. Let $\Pi$ be a probability measure on $(\Theta,\B)$.

\begin{assumption}
\label{assumption:main}
Assume
\begin{enumerate}
\item $\theta\mapsto P_\theta(A)$ is measurable for every $A\in\A$, and
\item $\theta\neq\theta'  \implies  P_\theta\neq P_{\theta'}$.
\end{enumerate}
\end{assumption}
In other words, condition (1) is that $\{P_\theta:\theta\in\Theta\}$ is a measurable family, and
condition (2) is that $\theta$ is identifiable: if $\theta,\theta'\in\Theta$ such that $\theta\neq\theta'$ then there exists $A\in\A$ such that $P_\theta(A)\neq P_{\theta'}(A)$.

As indicated in the introduction, Doob's theorem has two parts: the first concerning consistent estimation (Theorem~\ref{theorem:estimators}), and the second concerning posterior consistency (Theorem~\ref{theorem:posterior}). A word on notation: we use $\theta$ for fixed values in $\Theta$, and $\T$ for a random variable taking values in $\Theta$.
Let $\mu$ be the joint distribution of $((X_1,X_2,\ldots),\T)$ defined by letting $\T\sim\Pi$ and $X_1,X_2,\ldots|\T$ i.i.d.\ $\sim P_{\T}$.

\begin{theorem}[Doob's theorem for estimators]
\label{theorem:estimators}
Let $(X,\T)\sim \mu$, where $X =(X_1,X_2,\dotsc)$. If $g:\Theta\to \R$ is a measurable function such that $\E|g(\T)|<\infty$, and Assumption~\ref{assumption:main} holds, then 
$$\lim_{n\to\infty} \E(g(\T) \mid X_1,\dotsc,X_n) = g(\T) \,\,\, \mbox{a.s.}[\mu].$$
\end{theorem}
\begin{remark} The generalization to vector-valued $g:\Theta\to \R^k$ follows by applying the theorem to each coordinate separately. 
Also, note that the condition on $g$ is equivalent to requiring $g \in L^1(\Pi)$.
\end{remark}

The following corollary is equivalent to Theorem~\ref{theorem:estimators}, but is presented from a frequentist perspective.

\begin{corollary}
\label{corollary:estimators}
Suppose $g\in L^1(\Pi)$ and Assumption~\ref{assumption:main} holds. There exists $\Theta_0\subseteq\Theta$ with $\Pi(\Theta_0) = 1$ such that for all $\theta_0\in\Theta_0$, if $X_1,X_2,\dotsc\sim P_{\theta_0}$ i.i.d.\ then 
$$\lim_{n\to\infty} \E(g(\T) \mid X_1,\dotsc,X_n) = g(\theta_0) \,\,\, \mbox{a.s.}[P_{\theta_0}]$$ 
where the conditional expectation is computed under $\mu$.
\end{corollary}

\begin{theorem}[Doob's theorem for posterior consistency]
\label{theorem:posterior}
If Assumption~\ref{assumption:main} holds, then there exists $\Theta_1\subseteq\Theta$ with $\Pi(\Theta_1) = 1$ such that the posterior is consistent at every $\theta_1\in\Theta_1$. That is, for all $\theta_1\in\Theta_1$, if $X_1,X_2,\dotsc\sim P_{\theta_1}$ i.i.d.\ then for any neighborhood $B$ of $\theta_1$, we have 
$$\lim_{n\to\infty} \PP(\T\in B \mid X_1,\dotsc,X_n) = 1 \,\,\, \mbox{a.s.}[P_{\theta_1}]$$
where $\PP(\T\in B \mid X_1,\dotsc,X_n)=\E(I_B(\T) \mid X_1,\dotsc,X_n)$ is computed under $\mu$.
\end{theorem}

\section{Proofs of the main results}
\label{section:main-proofs}

First, we make some clarifying remarks and initial observations.
Let $\Xhat^\infty =\Xhat\times\Xhat\times\cdots =\{(x_1,x_2,\dotsc): x_1,x_2,\dotsc\in\Xhat\}$.
Then $\Xhat^\infty$ can be given a metric under which it is complete (as a metric space) and has the product topology. (For instance, if $d$ is the metric on $\Xhat$, then it can be verified that $D(x,y) =\sup_k \min\{1,d(x_k,y_k)\}/k$ is a metric on $\Xhat^\infty$ satisfying these properties.) % Munkres Exercise 21.3 and Theorem 43.4
It follows that $\Xhat^\infty$ is separable. % reference?
Therefore, $\XI=\X\times\X\times\cdots$ is a Borel measurable subset of a complete separable metric space, $\Xhat^\infty$. Further, the induced Borel $\sigma$-algebra on $\XI$ coincides with the product $\sigma$-algebra $\AI$, which is generated by sets of the form 
% \todo{verify}
$$A_1\times\cdots\times A_k\times\X\times\X\times\cdots =\{x\in\XI: x_1\in A_1,\dotsc,x_k\in A_k\} $$
where $A_1,\dotsc,A_k\in\A$ and $k\in\{1,2,\dotsc\}$. For each $\theta\in\Theta$, let $\PI$ be the product measure on $(\XI,\AI)$ obtained from $P_\theta$ --- that is, by the Kolmogorov extension theorem, let $\PI$ be the unique probability measure on $(\XI,\AI)$ such that
$$\PI(A_1\times\cdots\times A_k\times\X\times\X\times\cdots) = P_\theta(A_1)\cdots P_\theta(A_k) $$
for all $A_1,\dotsc,A_k\in\A$ and all $k\in\{1,2,\dotsc\}$.

The joint distribution $\mu$, defined in Section~\ref{section:theorem}, is equal to the probability measure on $(\XI\times\Theta,\AI\otimes\B)$ defined by
$$\mu(E) =\int_\Theta \PI(E_\theta) \, d\Pi(\theta)$$
for $E\in\AI\otimes\B$, where $E_\theta =\{x\in\XI:(x,\theta)\in E\}$ is the $\theta$-section of $E$. 
In order for this definition to make sense, we need $\theta\mapsto \PI(E_\theta)$ to be a measurable function for all $E\in\AI\otimes\B$; in Section~\ref{section:details}, we prove that this is indeed the case, as long as condition (1) of Assumption~\ref{assumption:main} holds.

Similarly, giving $\Xhat^\infty\times\widehat\Theta$ a metric that makes it a complete separable metric space with the product topology, it follows that $\XI\times\Theta$ is a Borel measurable subset of $\Xhat^\infty\times\widehat\Theta$, and the induced Borel $\sigma$-algebra on $\XI\times\Theta$ coincides with the product $\sigma$-algebra $\AI\otimes\B$.

\begin{proof}[\bf Proof of Theorem \ref{theorem:estimators}]
Since $\E|g(\T)|<\infty$, $\E(g(\T) \mid X_1,\dotsc,X_n)$ is a uniformly integrable martingale, and thus
$$\lim_{n\to\infty} \E(g(\T) \mid X_1,\dotsc,X_n) = \E(g(\T) \mid X) \,\,\,\,\, \mbox{a.s.}$$
by, e.g., \citet{Kallenberg}, 7.23. In Section \ref{section:recovery}, Theorem \ref{theorem:recovery}, we prove that there exists $f:\XI\to\Theta$ measurable such that $f(x) =\theta$ for $\mu$-almost all $(x,\theta)$. Defining $h = g\circ f$, we have $h:\XI\to\R$ measurable such that $h(x) = g(\theta)$ a.e.$[\mu]$, and hence $h(X) = g(\T)$ almost surely. Note that $\sigma(X) =\{A\times\Theta: A\in\AI\}\subseteq\AI\otimes\B$.  Now, $h(X)$ is a version of $\E(g(\T) | X)$, since it is $\sigma(X)$-measurable and for any $F\in\sigma(X)$ we have
$$\int_F g(\theta) d\mu(x,\theta) =\int_F h(x) d\mu(x,\theta). $$
Hence, $\E(g(\T) | X) = h(X) = g(\T)$ almost surely.
\end{proof}

A minor technical point regarding the preceding proof: even though $\E(g(\T) | X) = g(\T)$ almost surely, $g(\T)$ is not necessarily a version of $\E(g(\T) | X)$ since $g(\T)$ is not necessarily $\sigma(X)$-measurable. It is true that $h(x) = g(\theta)$ a.e.$[\mu]$ and $h$ is Borel measurable, however, this is not sufficient to guarantee $\sigma(X)$-measurability of $g(\T)$ since $(\XI\times\Theta,\AI\otimes\B,\mu)$ is not necessarily a complete measure space. When working in a probability measure space that is not necessarily complete, we follow the usual custom of using ``almost everywhere'' (and ``almost surely'' for random variables) to signify that there exists a measurable set of measure 1 on which the condition under consideration holds.

\begin{proof}[\bf Proof of Corollary \ref{corollary:estimators}]
By Theorem~\ref{theorem:estimators}, there exists a set $F\in\AI\otimes\B$ of $\mu$-probability 1 such that for all $(x,\theta)\in F$,
$$\lim_{n\to\infty} \E(g(\T) \mid x_1,\dotsc,x_n) = g(\theta).$$ 
Then $1 = \mu(F) =\int_\Theta \PI(F_\theta)\, d\Pi(\theta)$, and therefore $\int_\Theta (1-\PI(F_\theta)) d\Pi(\theta)= 0$. Since $1-\PI(F_\theta)\geq 0$ for all $\theta$, then $\PI(F_\theta) = 1$ a.e.$[\Pi]$ by, e.g., \citet{Kallenberg}, 1.24. In other words, there exists $\Theta_0\in\B$ with $\Pi(\Theta_0) = 1$ such that $\PI(F_\theta) = 1$ for all $\theta\in\Theta_0$.
\end{proof}

With this corollary in hand, we are well-positioned to prove the posterior consistency part of Doob's theorem.

\begin{proof}[\bf Proof of Theorem \ref{theorem:posterior}]
Let $\C\subseteq\B$ be a countable base for the topology of $\Theta$, using the fact that $\Theta$ is a separable metric space. (For example, take the open balls of radius $1/k$ about every point in a countable dense subset, for all $k\in\{1,2,\dotsc\}$.) For each $C\in\C$, choose $\Theta_C$ according to Corollary \ref{corollary:estimators} with $g(\theta) = I_C(\theta)$, so that $\Pi(\Theta_C) = 1$ and for each $\theta\in\Theta_C$, 
$$\E(I_C(\T) \mid X_1,\dotsc,X_n)\xrightarrow[n\to\infty]{} I_C(\theta)\,\,\, \mbox{a.s.}$$
when $X_1,X_2,\dotsc$ i.i.d.\ $\sim P_\theta$. Let 
$$\Theta_1 =\I_{C\in\C}\Theta_C.$$
Then $\Pi(\Theta_1) = 1$ since $\C$ is countable. Let $\theta_1\in\Theta_1$ and let $B$ be a neighborhood of $\theta_1$. Since $\C$ is a base, there exists $C\in\C$ such that $\theta_1\in C\subseteq B$. Therefore, if $X_1,X_2,\dotsc\sim P_{\theta_1}$ i.i.d.\ then 
$$\E(I_B(\T) \mid X_1,\dotsc,X_n)\geq\E(I_C(\T) \mid X_1,\dotsc,X_n)\xrightarrow[n\to\infty]{} I_C(\theta_1) = 1  \,\,\, \mbox{a.s.}$$
since $I_B(\theta)\geq I_C(\theta)$ and $\theta_1\in\Theta_1 \subseteq\Theta_C$; hence, $\PP(\T\in B \mid X_1,\dotsc,X_n)\to 1$ almost surely.
\end{proof}

\section{Measurable recovery}
\label{section:recovery}

In the proof of Theorem \ref{theorem:estimators}, we postponed a key step: the existence of $f$. Although it may at first appear to be a simple technical detail, this turns out to be surprisingly nontrivial --- or at least, it seems to necessitate the use of a highly nontrivial result (Theorem \ref{theorem:inverse}). 

Our proof below is modeled after \citet{Ghosh}, who provided a proof of the posterior consistency part of Doob's theorem.

\begin{theorem}[Measurable recovery]
\label{theorem:recovery}
Under the setup of Sections~\ref{section:theorem} and~\ref{section:main-proofs},
there exists $f:\XI\to\Theta$ measurable such that $f(x) =\theta$ a.e.$[\mu]$.
\end{theorem}
In other words, there exists $f:\XI\to\Theta$ measurable such that there is some $F\in\AI\otimes\B$ with $\mu(F) = 1$ satisfying $f(x) =\theta$ for all $(x,\theta)\in F$.

The difficult part of the theorem is ensuring that $f$ is measurable. Otherwise, it is relatively simple, since if $X_1,X_2,\dotsc\sim P_\theta$ i.i.d., then from the infinite sequence $X =(X_1,X_2,\dotsc)$ one can almost surely recover $P_\theta$ (using the strong law of large numbers and separability of $\X$), and from $P_\theta$ one can recover $\theta$ (by identifiability, condition (2) of Assumption~\ref{assumption:main}). This would determine a map from $X$ to $\theta$, but it comes with no guarantee of measurability.  Although measurability is often a technical detail that can be assumed away, one should bear in mind that when dealing with conditional expectations, measurability is of the essence.

The following innocent-looking but deep theorem plays a central role in the proof. 

\begin{theorem}[Lusin, Souslin, Kuratowski]
\label{theorem:inverse}
Let $\UU$ and $\V$ be Borel measurable subsets of complete separable metric spaces, and give them the induced Borel $\sigma$-algebras. If $E\subseteq\UU$ is measurable, and $g:\UU\to\V$ is a measurable function such that the restriction $g|_E$ is one-to-one, then $g(E)$ is measurable and $g|_E^{-1}:g(E)\to\UU$ is measurable.
\end{theorem}
\begin{proof} See \citet{Kechris}, Corollary 15.2, or \citet{parthasarathy2005probability}, Corollary 3.3, for the proof.
\end{proof}

\vspace{5pt}
\begin{proof}[\bf Proof of Theorem \ref{theorem:recovery}]
Let $\C\subseteq\A$ be a countable collection that generates $\A$, contains $\X$, and is closed under finite intersections. (For example, let $\C_0$ be a countable base for $\X$, and let $\C =\{\X\}\union\{A_1\n\cdots\n A_k:A_1,\dotsc,A_k\in\C_0,\,k = 1,2,\dotsc\}$.)  Let
$$ E =\I_{A\in\C} F^A $$
where
$$ F^A =\left\{(x,\theta)\in\XI\times\Theta:
  \lim_{n\to\infty}\frac{1}{n}\sum_{i = 1}^n I_A(x_i)  \mbox{ exists and equals } P_\theta(A)\right\}. $$
We will abuse notation slightly and write $\pi_x$ and $\pi_\theta$ for the projections from $\XI\times\Theta$ to $\XI$ and $\Theta$ respectively, i.e., $\pi_x(x,\theta) = x$ and $\pi_\theta(x,\theta) =\theta$. Note that $\pi_x$ and $\pi_\theta$ are measurable (by the definition of the product $\sigma$-algebra). We will show that:
\begin{enumerate}
\item[(a)] $E$ is measurable (i.e. $E\in\AI\otimes\B$),
\item[(b)] $\mu(E) = 1$, and
\item[(c)] for any $x\in\pi_x(E)$, there is a unique $\theta\in\Theta$ such that $(x,\theta)\in E$.
\end{enumerate}
Before proving (a), (b), and (c), let's see how to prove the result, assuming they are true. 
We will apply Theorem \ref{theorem:inverse} with $\UU =\XI\times\Theta$, $\V =\XI$, and $g =\pi_x$. Recall that $\XI$ and $\XI\times\Theta$ are Borel measurable subsets of complete separable metric spaces (in their respective Borel $\sigma$-algebras), and note that (c) means that the restriction $\pi_x|_E$ is one-to-one. Therefore, assuming (a) and (c), Theorem \ref{theorem:inverse} implies that $\pi_x(E)$ is measurable and the map $\phi:\pi_x(E)\to\XI\times\Theta$ defined by $\phi =\pi_x|_E^{-1}$ is measurable. Note that for any $(x,\theta)\in E$, we have $\phi(x) =(x,\theta)$ by construction. It is straightforward to show that $\phi$ has a measurable extension $\tilde\phi:\XI\to\XI\times\Theta$. Define $f:\XI\to\Theta$ by $f(x) =\pi_\theta(\tilde\phi(x))$. Then $f$ is measurable and for any $(x,\theta)\in E$ we have $f(x) =\pi_\theta(\tilde\phi(x)) =\pi_\theta(x,\theta) =\theta$. Along with (b), this will prove Theorem \ref{theorem:recovery}.

\vspace{10pt}
Now, we prove (a), (b), and (c).

\vspace{10pt}
{\boldface (a)} % $E$ is measurable: \\
Let $A\in\A$. Since $\pi_\theta$ is measurable and $\theta\mapsto P_\theta(A)$ is measurable (by condition (1) in Assumption~\ref{assumption:main}), then $(x,\theta)\mapsto P_\theta(A)$ is measurable. Similarly, for any $i\in\{1,2,\dotsc\}$, since $\pi_x$ is measurable and $x\mapsto x_i$ is measurable then $(x,\theta)\mapsto I_A(x_i)$ is measurable. Hence, the function
$$ h_n(x,\theta) :=\frac{1}{n}\sum_{i = 1}^n I_A(x_i)-P_\theta(A) $$
is measurable for any $n \in \{1,2,\dotsc\}$. In terms of $h_n$, the set $F^A$ defined above becomes
$$F^A =\left\{(x,\theta)\in\XI\times\Theta:\lim_{n\to\infty} h_n(x,\theta) \mbox{ exists and equals } 0\right\}.$$
It is straightforward to check that $F^A$ is measurable, and therefore $E =\I_{A\in\C} F^A$ is measurable. 

\vspace{10pt}
{\boldface (b)} % $\mu(E) = 1$: \\
For $\theta\in\Theta$ and $A\in\C$, the sections $E_\theta =\{x\in\XI:(x,\theta)\in E\}$ and $F^A_\theta =\{x\in\XI:(x,\theta)\in F^A\}$ are measurable, since from the proof of (a) we know that $E$ and $F^A$ are measurable. %(see e.g. Folland 2.34).
Noting that
$$ F^A_\theta =\left\{x\in\XI: \lim_{n\to\infty} \frac{1}{n}\sum_{i = 1}^n I_A(x_i) \mbox{ exists and equals }  P_\theta(A)\right\},$$
the strong law of large numbers implies that for any $\theta\in\Theta$ and any $A\in\C$, we have $\PI(F^A_\theta) = 1$. Hence, $\PI(E_\theta) = 1$ for any $\theta\in\Theta$, since $E_\theta =\I_{A\in\C} F^A_\theta$ (which follows from $E =\I_{A\in\C} F^A$) and $\C$ is countable. Therefore,
$$\mu(E) =\int_\Theta\PI(E_\theta)\, d\Pi(\theta) =\int_\Theta d\Pi(\theta) = \Pi(\Theta) = 1. $$

\vspace{10pt}
{\boldface (c)} % For any $x\in\pi_x(E)$ there is a unique $\theta\in\Theta$ such that $(x,\theta)\in E$: \\
First, we show that $\theta\neq\theta'\implies E_\theta\n E_{\theta'} =\emptyset$. Suppose $\theta,\theta'\in\Theta$ such that $\theta\neq\theta'$. Then $P_\theta\neq P_{\theta'}$ by identifiability (condition (2) of Assumption~\ref{assumption:main}). This implies that there exists $A\in\C$ such that $P_\theta(A)\neq P_{\theta'}(A)$ (by, e.g., \citet{Kallenberg}, 1.17) since $\C$ is a $\pi$-system containing $\X$ that generates $\A$. Therefore, $F^A_\theta\n F^A_{\theta'}=\emptyset$, and hence $E_\theta\n E_{\theta'} =\emptyset$.

Now, to prove (c), suppose $x\in\pi_x(E)$. Let $\theta,\theta'\in\Theta$ such that $(x,\theta),(x,\theta')\in E$. (There is at least one such $\theta$ by the definition of $\pi_x(E)$.) Then $x\in E_\theta$ and $x\in E_{\theta'}$, and therefore $\theta =\theta'$ (by the preceding paragraph). 

This proves (a), (b), and (c), completing the proof of Theorem \ref{theorem:recovery}.
\end{proof}

\section{Supporting results}
\label{section:details}

% In the preceding material, we omitted the proofs of some facts which are relatively straightforward. 
For completeness, in this section we provide some additional details.

\subsection{Joint measures}
\label{section:joint}

Let $(\X,\A)$ and $(\Theta,\B)$ be measurable spaces, let $P_\theta$ be a probability measure on $(\X,\A)$ for each $\theta\in\Theta$, and let $\Pi$ be a probability measure on $(\Theta,\B)$. 
We would like to define a joint probability measure $\nu$ on $(\X\times\Theta,\A\otimes\B)$ by
$$\nu(E) =\int_\Theta P_\theta(E_\theta) \, d\Pi(\theta)$$
for all $E\in\A\otimes\B$, where $E_\theta =\{x\in\X:(x,\theta)\in E\}$ is the $\theta$-section of $E$. 
However, in order for this definition to make sense, we need $\theta\mapsto P_\theta(E_\theta)$ to be a measurable function. 
% The following lemma shows that this is indeed the case, as long as $\theta\mapsto P_\theta(A)$ is measurable for every $A\in\A$. Initially, we treat the sample space $\X$ abstractly, and then we extend the result to the product space $\X^\infty$.

\begin{lemma}
\label{lemma:joint}
If $\theta\mapsto P_\theta(A)$ is measurable for every $A\in\A$, then $\theta\mapsto P_\theta(E_\theta)$ is measurable for every $E\in\A\otimes\B$.
\end{lemma}
\begin{proof}
We apply the $\pi-\lambda$ theorem; see, e.g., \citet{Kallenberg}, 1.1.
The collection of rectangles $\RC =\{A\times B: A\in\A,\,\, B\in\B\}$ is a $\pi$-system. 
Let $\L$ be the collection of sets $E\in\A\otimes\B$ such that $\theta\mapsto P_\theta(E_\theta)$ is measurable. We will show that $\L$ is a $\lambda$-system containing $\RC$. Since $\sigma(\RC) =\A\otimes\B$ then $\A\otimes\B\subseteq\L$ will follow by the $\pi-\lambda$ theorem. This will prove the lemma.

To see that $\RC\subseteq\L$, let $E\in\RC$ and write $E = A\times B$ with $A\in\A$, $B\in\B$. Then 
$$E_\theta =\branch{A}{\theta\in B}{\emptyset}{\theta\not\in B,}$$
so $P_\theta(E_\theta) = P_\theta(A) I_B(\theta)$, which is measurable as a function of $\theta$. Hence, $E\in\L$.

To show that $\L$ is a $\lambda$-system, we must show that $\X\times\Theta\in\L$ and $\L$ is closed under increasing unions and proper differences. Since $\RC\subseteq\L$, then $\X\times\Theta\in\L$ is immediate. If $E^1,E^2,\dotsc\in\L$ such that $E^1\subseteq E^2\subseteq\cdots$, and $E =\U_n E^n$, then 
$P_\theta(E_\theta) = P_\theta(\U_n E^n_\theta) = \lim_{n\to\infty} P_\theta(E^n_\theta)$, which is measurable as a function of $\theta$ (being a limit of measurable functions). If $E,F\in\L$ such that $E\subseteq F$ then $P_\theta((F\setminus E)_\theta)= P_\theta(F_\theta\setminus E_\theta) = P_\theta(F_\theta)-P_\theta(E_\theta)$, which is measurable as a function of $\theta$.
\end{proof}

\subsection{Measurable families are preserved under countable products}
\label{section:product}

Now, consider the situation in Sections \ref{section:theorem} and \ref{section:main-proofs}.  Since Assumption~\ref{assumption:main} ensures that $\theta\mapsto P_\theta(A)$ is measurable for all $A\in\A$, Lemma~\ref{lemma:joint} shows that $\theta\mapsto P_\theta(E_\theta)$ is measurable for all $E\in\A\otimes\B$. However, since we would like to define the joint measure as
$$\mu(E) =\int_\Theta \PI(E_\theta) \, d\Pi(\theta)$$
on $(\XI\times\Theta,\AI\otimes\B)$, what we really need is for $\theta\mapsto \PI(E_\theta)$ to be measurable for all $E\in\AI\otimes\B$. If we can show that $\theta\mapsto \PI(A)$ is measurable for all $A\in\AI$, then this will follow from Lemma~\ref{lemma:joint} simply by replacing $(\X,\A,P_\theta)$ with $(\XI,\AI,\PI)$.

\begin{lemma}
Under the setup of Sections~\ref{section:theorem} and \ref{section:main-proofs}, $\,\theta\mapsto \PI(A)$ is measurable for every $A\in\AI$.
\end{lemma}
\begin{proof}
This is another application of the indispensable $\pi-\lambda$ theorem.
For our $\pi$-system, we take
$$\RC=\{A_1\times\cdots\times A_k\times\X\times\X\times\cdots: A_1,\dotsc,A_k\in\A,\,\,k = 1,2,\dotsc\}.$$
Let $\L$ be the collection of sets $A\in\AI$ such that $\theta\mapsto \PI(A)$ is measurable. We will show that $\L$ is a $\lambda$-system containing $\RC$. Since $\sigma(\RC) =\AI$, then $\AI\subseteq\L$ will follow by the $\pi-\lambda$ theorem, proving the desired result.

To see that $\RC\subseteq\L$, note that if $A_1,\dotsc,A_k\in\A$ then
$$\PI(A_1\times\cdots\times A_k\times\X\times\X\times\cdots) = P_\theta(A_1)\cdots P_\theta(A_k),$$
which is measurable as a function of $\theta$, since $\theta\mapsto P_\theta(A_i)$ is measurable for each $i = 1,\dotsc,k$ by Assumption~\ref{assumption:main}.

To show that $\L$ is a $\lambda$-system, we must show that $\XI\in\L$ and $\L$ is closed under increasing unions and proper differences. Since $\RC\subseteq\L$, we have $\XI\in\L$. If $A_1,A_2,\dotsc\in\L$ such that $A_1\subseteq A_2\subseteq\cdots$, and $A =\U_i A_i$, then 
$\PI(A) = \lim_{i\to\infty} \PI(A_i)$, which is measurable as a function of $\theta$ (being a limit of measurable functions). Lastly, if $A,B\in\L$ such that $A\subseteq B$ then $\PI(B\setminus A) = \PI(B)-\PI(A)$ is measurable as a function of $\theta$.
\end{proof}

% \section*{Acknowledgments}

\bibliography{references}

\begin{thebibliography}{12}
\providecommand{\natexlab}[1]{#1}
\providecommand{\url}[1]{\texttt{#1}}
\expandafter\ifx\csname urlstyle\endcsname\relax
  \providecommand{\doi}[1]{doi: #1}\else
  \providecommand{\doi}{doi: \begingroup \urlstyle{rm}\Url}\fi

\bibitem[Barron et~al.(1999)Barron, Schervish, and
  Wasserman]{barron1999consistency}
A.~Barron, M.~J. Schervish, and L.~Wasserman.
\newblock The consistency of posterior distributions in nonparametric problems.
\newblock \emph{The Annals of Statistics}, 27\penalty0 (2):\penalty0 536--561,
  1999.

\bibitem[Doob(1949)]{Doob_1949}
J.~L. Doob.
\newblock Application of the theory of martingales.
\newblock In \emph{Actes du Colloque International Le Calcul des
  Probabilit\'{e}s et ses applications (Lyon, 28 Juin -- 3 Juillet, 1948)},
  pages 23--27. Paris CNRS, 1949.
\newblock (Doob's original paper can be found at the end of a historical
  account entitled ``Doob at Lyon'', currently available at:
  \url{www.jehps.net/juin2009/Locker.pdf}.).

\bibitem[Freedman(1963)]{freedman1963asymptotic}
D.~A. Freedman.
\newblock On the asymptotic behavior of {B}ayes' estimates in the discrete
  case.
\newblock \emph{The Annals of Mathematical Statistics}, pages 1386--1403, 1963.

\bibitem[Ghosal and van~der Vaart(2017)]{ghosal2017fundamentals}
S.~Ghosal and A.~van~der Vaart.
\newblock \emph{Fundamentals of Nonparametric Bayesian Inference}.
\newblock Cambridge University Press, 2017.

\bibitem[Ghosh and Ramamoorthi(2003)]{Ghosh}
J.~K. Ghosh and R.~V. Ramamoorthi.
\newblock \emph{Bayesian {N}onparametrics}.
\newblock Springer-Verlag, New York, 2003.

\bibitem[Kallenberg(2002)]{Kallenberg}
O.~Kallenberg.
\newblock \emph{Foundations of {M}odern {P}robability (Second Edition)}.
\newblock Springer-Verlag, New York, 2002.

\bibitem[Kechris(1995)]{Kechris}
A.~S. Kechris.
\newblock \emph{Classical {D}escriptive {S}et {T}heory}.
\newblock Springer-Verlag, New York, 1995.

\bibitem[Lijoi et~al.(2004)Lijoi, Pr{\"u}nster, and Walker]{lijoi2004extending}
A.~Lijoi, I.~Pr{\"u}nster, and S.~G. Walker.
\newblock Extending {D}oob's consistency theorem to nonparametric densities.
\newblock \emph{Bernoulli}, 10\penalty0 (4):\penalty0 651--663, 2004.

\bibitem[Nobile(1994)]{Nobile_1994}
A.~Nobile.
\newblock \emph{Bayesian Analysis of Finite Mixture Distributions.}
\newblock PhD thesis, Department of Statistics, Carnegie Mellon University,
  Pittsburgh, PA, 1994.

\bibitem[Parthasarathy(2005)]{parthasarathy2005probability}
K.~R. Parthasarathy.
\newblock \emph{Probability Measures on Metric Spaces}.
\newblock American Mathematical Society, Providence, RI, 2005.

\bibitem[Schwartz(1965)]{schwartz1965bayes}
L.~Schwartz.
\newblock On {B}ayes procedures.
\newblock \emph{Probability Theory and Related Fields}, 4\penalty0
  (1):\penalty0 10--26, 1965.

\bibitem[Van~der Vaart(1998)]{van1998asymptotic}
A.~W. Van~der Vaart.
\newblock \emph{Asymptotic Statistics}.
\newblock Cambridge University Press, 1998.

\end{thebibliography}
\bibliographystyle{abbrvnat}

\end{document}